\documentclass{amsart}
\usepackage[style=alphabetic,giveninits=true]{biblatex}
\usepackage{amsfonts}
\usepackage{verbatim}
\usepackage{amsmath}
\usepackage[colorlinks]{hyperref}
\usepackage[utf8]{inputenc}
\usepackage{todonotes}
\newcommand{\F}{{\mathcal F}}
\newcommand{\R}{{\mathbb R}}

\DeclareMathOperator{\Var}{Var}

\newcommand{\0}{\textbf{0}}
\renewcommand{\P}{{\mathbb P}}
\newcommand{\tmix}{t_{{\rm mix}}}
\newcommand{\ep}{\varepsilon}
\newcommand{\E}{{\mathbb E}}
\newcommand{\Z}{{\mathbb Z}}
\newtheorem{theorem}{Theorem}
\newtheorem{lem}{Lemma}
\theoremstyle{remark}
\newtheorem{rmk}{Remark}

\title[A fast mixing hypercube chain]{Fast mixing of a randomized shift-register Markov chain} \author{David A. Levin}
\email{dlevin@uoregon.edu} \author{Chandan Tankala}
\email{chandant@uoregon.edu} \address{Department of Mathematics, The
  University of Oregon, Eugene, OR 97403-1222 USA}



\begin{document}
\begin{abstract}
  We present a Markov chain on the $n$-dimensional hypercube
  $\{0,1\}^n$ which satisfies $\tmix^{(n)}(\ep) = n[1 + o(1)]$.  This
  Markov chain alternates between random and deterministic moves and
  we prove that the chain has cutoff with a window of size at most
  $O(n^{0.5+\delta})$ where $\delta>0$.  The deterministic moves
  correspond to a linear shift register.
\end{abstract}
\maketitle

\section{Introduction}
\label{section:1}

Developing Markov chain Monte Carlo (MCMC) algorithms with fast mixing times remains a problem of practical importance.   One would like to make computationally tractable modifications to existing chains which decrease the time required to obtain near equilibrium samples. 

The \emph{mixing time} of an ergodic finite Markov chain
$(X_t)$ with stationary distribution $\pi$ is defined as
\begin{equation}
    \tmix(\ep) = \min\Bigl\{ t \geq 0 \,:\, \max_{x} \|
        \P_x(X_t \in \cdot ) - \pi \|_{{\rm TV}} < \ep \Bigr\} \,,
\end{equation}
and we write $\tmix = \tmix(1/4)$.

A theoretical algorithm for chains with uniform stationary distribution is analyzed in Chatterjee and Diaconis \cite{chatterjee2020speeding}. They proposed chains
that alternate between random steps made according to a probability transition matrix and deterministic steps defined by a bijection $f$ on the state space.  Supposing the state-space has size $n$, the transition matrix satisfies a one-step reversibility condition, and
$f$ obeys an \emph{expansion condition}, 
they proved that $\tmix = O(\log n)$. However, they note that finding an explicit bijection $f$ satisfying the expansion condition can be difficult even for simple state spaces like $\Z_n$.

In this paper, we analyze a Markov chain on the hypercube $\{0,1\}^n$ of the form $P\Pi$ for an explicit $\Pi$, where $P$ corresponds to the usual lazy random walk on $\{0,1\}^n$.  This chain may be of independent interest, as the deterministic transformation $f$ on the state space is a ``shift register'' operator. Such shift registers have many applications in cryptography, psuedo-random number generation, coding, and other fields. See, for example, \textcite{G:SRS} for background on shift registers.
 
The \emph{lazy random walk} on $\{0,1\}^n$ makes transitions as follows: when the current state is $x$, a coordinate from
$i \in \{1,2,\ldots,n\}$ is generated uniformly at random, and an independent random bit $R$ is added (mod $2$) to the bit $x_i$ at coordinate $i$.  The new state obtained is thus
\begin{equation} \label{eq:P} x \mapsto x' = (x_1, \ldots, x_i \oplus   R, \ldots, x_n) \,.
\end{equation}
We will denote the transition matrix of this chain by $P$. For a chain with transition probabilities $Q$ on $S$ and stationary distribution $\pi$, let
\[
  d(t) = d_n(t) = \max_{x \in S} \| Q^t(x,\cdot) - \pi
  \|_{{\rm TV}} \,.
\]
A sequence of chains indexed by $n$ has a \emph{cutoff} if, for $t_n := \tmix^{(n)}$, there exists a \emph{window sequence} $\{w_n\}$ with $w_n = o(t_n)$ such that
\begin{align*}
  \limsup_{n \to \infty} d(t_n + w_n) & = 0 \\
  \liminf_{n \to \infty} d(t_n - w_n) & = 1 \,.
\end{align*}
For background on mixing times, cutoff, and related material, see, for example, \textcite{levin2017markov}.

It is well-known that for the lazy random walk on $\{0,1\}^n$,
\[
  \tmix(\ep) = \frac{1}{2}n \log n[1 + o(1)] \,,
\]
with a cutoff.  (See \textcite{DGM} where precise information on the total variation distance is calculated. The difference of a factor of $2$ above comes from the laziness in our version.)

A natural deterministic ``mixing'' transformation on $\{0,1\}^n$ is the ``linear shift register'' which takes the xor sum of the bits in the current word $x = (x_1,\ldots,x_n)$ and appends it to the right-hand side, dropping the left-most bit:
\begin{equation} \label{eq:fdef} x \mapsto f(x) = \Bigl(x_2, \ldots,
  x_{n-1}, \oplus_{i=1}^n x_{i}\Bigr) \,.
\end{equation}
 
Let $\Pi$ denote the permutation matrix corresponding to this transformation, so that
\[
  \Pi_{i,j} =
  \begin{cases}
    1 & \text{if } j = f(i) \\
    0 & \text{else}.
  \end{cases}
\]
 
The chain studied in the sequel has the transition matrix $Q_1 = P\Pi$, whose dynamics are simply described by combining the stochastic operation \eqref{eq:P} with the deterministic $f$ in \eqref{eq:fdef}:
\[
  x \mapsto x' \mapsto f(x') \,.
\]

Let $\log $ stand for the natural logarithm. The main result here is:
\begin{theorem} \label{thm:main} For the chain $Q_1$, \hspace{1.2in}
  \begin{itemize}
  \item[(i)] For $n \geq 5$,
    \[
      d_n(n+1) \leq \frac{2}{n} \,.
    \]
  \item[(ii)] For any $1/2 < \alpha < 1$, if $t_n = n - n^{\alpha}$
    \[
      d_n(t_n) \geq \| Q_1^{t_n}(0, \cdot) - \pi \|_{{\rm TV}} \geq 1 -
      o(1) \,.
    \]
  \end{itemize}
  Thus, the sequence of chains has a cutoff at time $n$ with a window of at most size $n^{1/2 + \delta}$ for any $\delta > 0$.
\end{theorem}

\begin{rmk}
  If the transformation $f$ obeys the expansion condition of
  \textcite{chatterjee2020speeding}, then the results therein yield a mixing time of order $n$.  We were unable to directly verify that $f$ does obey this condition.  Moreover, the result in Theorem \ref{thm:main} establishes the stronger cut-off property.
\end{rmk}
     
 
\begin{rmk}
  Obviously a simple way to exactly randomize $n$ bits in exactly $n$ steps is to simply randomize in sequence, say from left to right, each bit. This is called \emph{systematic scan}; Systematic scans avoid an extra factor of $\log n$ needed for \emph{random updates}
  to touch a sufficient number of bits.  (A ``coupon-collector'' argument shows that to touch all but $O(\sqrt{n})$ bits using random updates -- enough to achieve small total-variation distance from uniform -- order $n\log n$ steps are required.) Thus, clearly our interest in analyzing this chain is not for direct simulation of $n$
  independent bits!  Rather, we are motivated both by the potential for explicit deterministic moves to speed-up of Markov chains, and also by this particular chain which randomizes the well-known shift-register dynamical system.  
\end{rmk}

This paper is organised as follows. In Section \ref{section:2} we review some related results. The upper bound in Theorem \ref{thm:main} is proved in Section \ref{section:5}, and the lower bound is established in Section \ref{section:4}. In section \ref{section:6}, a chain is analyzed that is similar to the chain of Theorem \ref{thm:main}, but always updates the same location.

\section{Related previous work} \label{section:2}

\subsection{Markov chains on hypercube}


Previous work on combining deterministic transformation with random moves on a hypercube is \textcite{diaconis1992affine}.  They study the walk $\{X_t\}$ described by
$X_{t+1}=AX_t+\epsilon_{t+1}$ where $A$ is a $n\times n$ lower
triangular matrix and $\epsilon_t$ are i.i.d.\ vectors having the following distribution: The variable $\epsilon_{t}=\mathbf{0}$ with probability $\theta\neq \frac{1}{2}$ while $\epsilon_{t}=e_1$ with probability $1-\theta$. Here, $\mathbf{0}$ is a vector of zeros,
$\mathbf{e_1}$ is a vector with a one in the first coordinate and zeros elsewhere. Fourier analysis is used to show that $O(n\log n)$ steps are necessary and sufficient for mixing, and they prove a sharp result in both directions. This line of work is a specific case of a random walk on a finite group $G$ described as $X_{t+1}=A(X_t)\epsilon_{t+1}$, where $A$ is an automorphism of $G$ and $\epsilon_1,\epsilon_2,\cdots$ are i.i.d. with some distribution $\mu$ on $G$. In the case of \textcite{diaconis1992affine}, $G=\Z_2^n$ and the automorphism $A$ is a matrix. By comparison, our chain studied here mixes in only $n(1 + o(1))$ steps.

Another relevant (random) chain on $\{0,1\}^n$ is analyzed by \textcite{wilson1997random}.  A subset of size $p$ from $\Z_2^n$ is chosen uniformly at random, and the graph $G$ with vertex set $\Z_2^n$ is formed which contains an edge between vertices if and only if their difference is in $S$.   
\textcite{wilson1997random} considered the random walk on the random graph $G$.
It is shown that if $p=cn$ where $c>1$ is a constant, then the mixing time is linear in $n$ with high probability (over the choice of $S$) as $n \to \infty$.
This Markov chain depends on the random environment to produce the speedup.

Finally, another example of cutoff for a Markov chain on a hypercube is \textcite{ben2018cutoff}. This random walk moves by picking an ordered
pair $(i, j)$ of distinct coordinates uniformly at random and adding the bit at location $i$ to the bit at location $j$, modulo $2$. They proved that this Markov chain has cutoff at time $\frac{3}{2}n\log n$ with window of size $n$, so the mixing time is the same order as that of the ordinary random walk.

\subsection{Related approaches to speeding up mixing}
 
Ben-Hamou and Peres \cite{ben2021cutoff} refined
the results of \textcite{chatterjee2020speeding}, 
proving further that under mild assumptions on $P$ ``typical'' $f$ yield a mixing time of order $\log n$ with a cutoff.   
In particular, they show that 
if a permutation matrix $\Pi$ is selected uniformly at random, then 
the (random) chain $Q=P\Pi$ has 
cutoff at $\frac{\log  n}{\textbf{h}}$ with high probability (with respect to the selection of $\Pi$).    Here \textbf{h} is the entropy rate of $P$ and
$n$ is the size of state space.
Like the chain in \textcite{wilson1997random}, the random environment is critical to the analysis.   However, in specific applications, one would like to know an explicit deterministic permutation $\Pi$ that mixes in $O(\log n)$ and does not require storage of the matrix $\Pi$, particularly when the state space increases exponentially with $n$.

A method for speeding up mixing called \emph{lifting} was introduced by \textcite{diaconis2000analysisnonrev}.
The idea behind this technique is to create ``long cycles'' and introduce non-reversibility.    For example, for simple random walk on the $n$-path, 
the mixing time of the lifting is $O(n)$, whereas the mixing time on the path is $\Theta(n^2)$.   Thus this method can provide a speed-up of square root of the mixing time of the original chain.    \textcite{chen1999lifting} give an explicit lower bound on the mixing time of the lifted chain in terms of the original chain. 
The chain we study has a similar flavor in that the transformation $f$ creates non-reversibility and long cycles.   

Another related speed-up technique is \emph{hit and run},  which introduces non-local  moves in a chosen direction. See the survey by \textcite{andersen2007hit}. \textcite{boardman2020hit} is a recent application of a top-to-random shuffle where it is shown that a speedup in mixing by a constant factor can be obtained for the $L^2$ and sup-norm. Jensen and Foster \cite{jensen2014level} have used this method to sample from  high-dimensional and multi-modal posterior distributions in Bayesian models, and compared that with Gibbs and Hamiltonian Monte Carlo algorithms. In the physics literature, non-reversible chains
are constructed from a reversible chains without augmenting the state space (in contrast to lifting) by introducing \emph{vorticity},
which is similar in spirit to the long cycles generated by lifting;
see, for example, Bierkens \cite{bierkens2016non}, which analyzes a non-reversible version of Metropolis-Hastings.

As mentioned above, there are obvious 
other methods to obtain a fast uniform sample from $\{0,1\}^n$,  In particular systematic scan, which generates an exact sample in precisely $n$ steps! See \textcite{diaconis2000analysis} for a comparison of systematic and random scans on different finite groups.

\section{Upper bound of Theorem \ref{thm:main}}
\label{section:5}

The proof is based on Fourier analysis on $\Z_2^n$.
Let A be a matrix defined as follows
\begin{equation} \label{eq:Adefn}
  A:=\begin{bmatrix}0&1&0&\cdots 0\\
    0&0&1&\cdots 0\\
    \vdots&\ddots \\
    0&0&0&\cdots 1 \\
    1&1&1&\cdots1
  \end{bmatrix}_{n\times n} \,.
\end{equation}
 
Let $\{\epsilon_i\}$ be i.i.d.\ random vectors and have the following
distribution
\begin{align}
  \label{eq:5.1}
  \epsilon_i = \begin{cases}
    \textbf{0} \quad \text{with probability } \frac{1}{2}  \\
    e_1 \quad \text{with probability } \frac{1}{2n}  \\
    \vdots \\
    e_n \quad \text{with probability } \frac{1}{2n}  \,,\\
  \end{cases}
\end{align}
where $e_i = (0,\ldots,\underbrace{1}_{\text{$i$-th
    place}},\ldots,0)$.
The random walk $X_t$ with transition matrix $Q_1$ and
$X_0=x$ can be described as
\[
  X_t = A(X_{t-1}\oplus \epsilon_{t}) \,.
\]
The matrix arithmetic above is all modulo $2$.  Induction shows that
\begin{equation}
  X_t = \left(\sum_{j=1}^t A^{t-j+1}\epsilon_j\right) \oplus A^tx \label{eq:5.2} \,,
\end{equation}
again where the matrix multiplication and vector sums are over the
field $\mathbb{Z}_2$.

\begin{lem}
  The matrix $A$ in \eqref{eq:Adefn} satisfies $A^{n+1}=I_{n\times
    n}$.
  \label{lemma:5.1}
\end{lem}
 
\begin{proof}
  Note that
  \begin{align*}
    Ae_1 &= e_n \\
    A^2e_1 = A(Ae_1)&= e_{n}+e_{n-1} \\
    A^3e_1 = A(A^2e_1) &= e_{n-1}+e_{n-2} \\
         &\vdots \\
    A^ne_1 &= e_2+e_1
  \end{align*}
  This implies that
  \[
    A^{n+1}e_1=A(A^ne_1)=A(e_2+e_1)=e_1 \,.
  \]
  The reader can check similarly that $A^{n+1}e_j=e_j$ for
  $2\leq j \leq n$.
\end{proof}

For $x,y\in \Z_2^n$, the Fourier transform of $Q_1^t(x,\cdot)$ at $y$
is defined as
\begin{align}
  \widehat{Q_1^t}(x,y)&:=\sum_{z\in \mathbb{Z}_2^n} (-1)^{y \cdot z}Q_1^t(x,z) \nonumber \\
                      &= \E[(-1)^{y \cdot X_{t}}] \label{eq:5.3} \\
                      &= (-1)^{y \cdot A^tx}\prod_{j=1}^{t} \E\left[ (-1)^{y \cdot A^{t-j+1}\epsilon_j}\right]\,. \label{eq:5.4}
\end{align}
The product $x \cdot y$ is the inner product $\sum_{i=1}^n x_i y_i$.
The equality \eqref{eq:5.4} follows from plugging \eqref{eq:5.2} into
\eqref{eq:5.3} and observing that $\epsilon_j$ are independent. The
following Lemma bounds the total variation distance; this is proved in
\textcite[Lemma 2.3]{diaconis1992affine}.

\begin{lem}
  \label{lemma:5.2}
  $\|Q_1^t(x,\cdot)-\pi\|_{{\rm TV}}^2 \leq \frac{1}{4}\sum_{y\neq 0}
  \left( \widehat{Q_1^t}(x,y)\right) ^2$.
\end{lem}

We will need the following Lemma to prove Theorem \ref{thm:main}(i):

\begin{lem}
  \label{lemma:5.6}
  Let
  \begin{equation} \label{eq:hnk} h(n,k) = \binom{n}{k} \left(1 -
      \frac{k}{n}\right)^{2n-2k} \left( \frac{k}{n} \right)^{2k}
  \end{equation}
  If $2 \leq k \leq n-2$ and $n > 5$, then $h(n,k) \leq 1/n^2$ and $h(n,n-1)\leq 1/n$.
\end{lem}
\begin{proof}
  We first prove this for $2 \leq k \leq n-2$.  For
  $x,y \in \mathbb{Z}^+$, define
  \begin{align*}
    K(x,y) & :=\log\frac{\Gamma(x + 1)}{\Gamma(y + 1)\Gamma(x - y + 1)}+(2x-2y)\log\left( 1-\frac{y}{x}\right)\\
           & \quad +2y\log\left( \frac{y}{x}\right)+2\log(x) \,,
  \end{align*}
  where $\Gamma$ is the Gamma function.  We prove that, if
  $2 \leq y \leq x-2$ and $x > 5$, then $K(x,y) < 0$.  Since
  $K(n,k) = \log \left(h(n,k) n^2\right)$, this establishes the lemma.

%
%
%

  Let $\psi(x):=\frac{d\log \Gamma(x)}{dx}$. Then
  \begin{align}
    \frac{\partial^2 K}{\partial y^2} &= -\psi'(y+1)-\psi'(x-y+1)+\frac{2}{x-y}+\frac{2}{y} \nonumber \\
                                      &> -\frac{1}{y+1}-\frac{1}{(y+1)^2}-\frac{1}{(x-y+1)}-\frac{1}{(x-y+1)^2}+\frac{2}{y}+\frac{2}{x-y} \label{eq:5.7} \\
                                      &> 0 \,.\label{eq:5.8}
  \end{align}
  The inequality \eqref{eq:5.7} follows from Guo and Qi \cite[Lemma
  1]{guo2013refinements}, which says that
  $\psi'(x)<\frac{1}{x}+\frac{1}{x^2}$ for all $x>0$.

  The second inequality follows since
  $2/y - (y+1)^{-1} - (y+1)^{-2} > 0$, and we can apply this again
  substituting in $x-y$ for $y$.  Thus $K(x,\cdot)$ is a convex
  function for all $x$. Also
  \begin{align}
    K(x,2)=K(x,x-2)&=\log\left(\frac{x(x-1)}{2} \right)+2(x-2)\log\left(1-\frac{2}{x} \right)+4\log\left(\frac{2}{x}\right)+2\log x \nonumber \\
                   &= \log \left( \frac{8x(x-1)}{x^2}\right)+2(x-2)\log\left( 1-\frac{2}{x}\right) \nonumber \\
                   &< \log (8)-\frac{4(x-2)}{x} \label{eq:5.9} \\
                   &<0  \,,\label{eq:5.10} 
  \end{align}
  for $x > 5$.  The inequality \eqref{eq:5.9} follows from
  $\log(1-u)<-u$. Equations \eqref{eq:5.10} and \eqref{eq:5.8} prove this lemma for $2\leq k\leq n-2$. Finally, $h(n,n-1) \leq 1/n \iff  nh(n,n-1)\leq 1$ which is true because one can verify that $\frac{d}{dn}nh(n,n-1)<0$ for $n\geq 5$ and $nh(n,n-1)<1$ for $n=5$.
\end{proof}
\begin{proof}[Proof of Theorem \ref{thm:main}(i)]
  Let $y=(y_1,y_2,\cdots,y_n)\in \Z_n^2$.  First,
  \begin{align*}
    \widehat{Q_1^{n+1}}(x,y) & =(-1)^{y \cdot A^{n+1}x}\prod_{j=1}^{n+1} \E\left[ (-1)^{y\cdot A^{n-j+2}\epsilon_j}\right]\\
                             & =(-1)^{y \cdot Ix}\prod_{j=1}^{n+1} \E\left[ (-1)^{y \cdot A^{n-j+2}\epsilon_j}\right] \,,
  \end{align*}
  which follows from \eqref{eq:5.4} and Lemma \ref{lemma:5.1}. Note
  that the first factor in this product is
  \[
    \left( \frac{1}{2} + \frac{1}{2n}\left[
        (-1)^{y_1}+(-1)^{y_2}+(-1)^{y_3}+\cdots+(-1)^{y_n}\right]\right)
    \,,
  \]
  which follows from \eqref{eq:5.1} and Lemma \ref{lemma:5.1}. We can
  similarly find other factors in the product, which gives
  \begin{align}
    \widehat{Q_1^{n+1}}& (x,y)  = (-1)^{x \cdot y}\nonumber \\
                       & \quad \times \left( \frac{1}{2} + \frac{1}{2n}\left[ (-1)^{y_1}+(-1)^{y_2}+(-1)^{y_3}+\cdots+(-1)^{y_n}\right]\right) 
                         \label{eq:secondcase} \\
                       & \quad \times \left( \frac{1}{2} + \frac{1}{2n}\left[ (-1)^{y_1}+(-1)^{y_1+y_2}+(-1)^{y_1+y_3}+\cdots+(-1)^{y_1+y_n}\right]\right)
                         \label{eq:fac1}\\
                       & \quad \times \left(\frac{1}{2} + \frac{1}{2n}\left[ (-1)^{y_2}+(-1)^{y_2+y_1}+(-1)^{y_2+y_3}+\cdots+(-1)^{y_2+y_n}\right]\right)  \\
                       &\quad \vdots  \\
                       &\quad \times \left(\frac{1}{2} + \frac{1}{2n}\left[ (-1)^{y_n}+(-1)^{y_n+y_1}+(-1)^{y_n+y_2}+\cdots+(-1)^{y_n+y_{n-1}}\right]\right)\label{eq:5.5}
  \end{align}

  Observe that $\widehat{Q_1^{n+1}}(x,y)=0$ for all $y\in \Z_2^n$ such that $W(y)\in \{1,n\}$ where $W(y)$ is the Hamming weight of $y\in \Z_2^n$. If $W(y)=1$, then one of the factors
  displayed on line \eqref{eq:fac1} through line \eqref{eq:5.5} is
  zero.  If $W(y)=n$, then the factor on line \eqref{eq:secondcase} is zero. If we fix a $2\leq j\leq n-1$ and look at
  all $y\in \Z_2^n$ with $W(y)=j$, then $[\widehat{Q_1^{n+1}}(x,y)]^2$
  is the same for all such $y$, since the expression above is
  invariant over permutation of coordinates, once the first factor is
  squared.

  If $y=(\underbrace{1,1,\ldots,1}_{\text{$k$ ones}},0,\cdots,0)$
  where $2\leq k\leq n-1$, then
  \[
    \widehat{Q_1^{n+1}}(x,y)=(-1)^{\left(\sum_{i=1}^kx_i\right)}\left(1-\frac{k}{n}\right)^{n-k+1}\left(\frac{k-1}{n}\right)^k
    \,.
  \]
  This holds because factors $3$ through $(k+2)$ are equal to
  $\frac{k-1}{n}$ and all factors except the first and second one are
  equal to $\frac{n-k}{n}$. To see this, note that factor $2$ is equal
  to

\begin{align*}
  &\frac{1}{2} + \frac{1}{2n}\left[ (\underbrace{-1-1\ldots-1}_{\text{$k$ negative ones}}) +(\underbrace{1+1\ldots+1}_{\text{$n-k$ positive ones}}) \right] \\
  & = \frac{1}{2} + \frac{1}{2n}\left[ -k + (n-k) \right] 
    = \frac{n-k}{n}
\end{align*}

Factors $3$ through $(k+2)$ are equal to
\begin{align*}
  &\frac{1}{2} + \frac{1}{2n}\left[ -1+ (\underbrace{1+1\ldots+1}_{\text{$k-1$ positive ones}}) + (\underbrace{-1-1\ldots-1}_{\text{$n-k$ negative ones}})\right] \\
  & = \frac{1}{2} + \frac{1}{2n}\left[-1+k-1-(n-k)\right] = \frac{k-1}{n} \,.
\end{align*}

Factors $(k+3)$ through $(n+2)$ are equal to
\begin{align*}
  &\frac{1}{2} + \frac{1}{2n}\left[ 1+ (\underbrace{-1-1\ldots-1}_{\text{$k$ negative ones}}) + (\underbrace{1+1\ldots+1}_{\text{$n-k-1$ positive ones}}) \right] \\
  & = \frac{1}{2} + \frac{1}{2n}\left[1-k+n-k-1\right] = \frac{n-k}{n}
\end{align*}

Thus,
\begin{align}
  \sum_{y\neq 0}\left(\widehat{Q_1^{n+1}}(x,y)\right)^2 &= \sum_{k=2}^{n-1} \binom{n}{k}\left(1-\frac{k}{n}\right)^{2n-2k+2} \left(\frac{k-1}{n}\right)^{2k} \nonumber \\
                                                        &\leq \sum_{k=2}^{n-1} \binom{n}{k}\left(1-\frac{k}{n}\right)^{2n-2k} \left(\frac{k}{n}\right)^{2k} \label{eq:5.6} \,.
\end{align}

We finally analyze the terms in the sum of \eqref{eq:5.6}. Note that
\[
  \binom{n}{k} \left(1-\frac{k}{n}\right)^{2n-2k}
  \left(\frac{k}{n}\right)^{2k}\leq \frac{1}{n^2}
\]
for $2\leq k\leq n-2$ and $n>5$ and $h(n,n-1)\leq 1/n$, by Lemma \ref{lemma:5.6}.
Thus,
%
\begin{equation} \label{eq:finbound} \sum_{k=2}^{n-1}
  \binom{n}{k}\left(1-\frac{k}{n}\right)^{2n-2k}
  \left(\frac{k}{n}\right)^{2k} \leq \frac{n-3}{n^2} + \frac{1}{n} \leq \frac{2}{n}
  \,.
\end{equation}
Lemma \ref{lemma:5.2}, with \eqref{eq:5.6} and the bound
\eqref{eq:finbound}, establishes the upper bound in Theorem
\ref{thm:main}.
\end{proof}

\section{Lower bound of Theorem \ref{thm:main}}
\label{section:4}

Let $\{U_t\}$ be the sequence of coordinates used to update the chain,
and let $\{R_t\}$ be the sequence of random bits used to update.
Thus, at time $t$, coordinate $U_t$ is updated using bit $R_t$.  Both
sequences are i.i.d.  Let $\F_t$ be the $\sigma$-algebra generated by
$(U_1,\ldots,U_t)$ and $(R_1,\ldots,R_t)$.  Let
$X_t = (X_t^{(1)}, \ldots, X_t^{(n)})$ be the chain with transition
matrix $Q$, at time $t$.  The proof is based on the distinguishing
statistic $W_t = \sum_{i=1}^n X^{(i)}_t$, the Hamming weight at time
$t$.

First, observe that

\[
  \P\left(\left.  X_{t+1}^{(n)} = 1 \right| \F_t \right) = \frac{1}{2}
  \,,
\]
because if the state at time $t$ is $x=(x_1,x_2,\cdots,x_n)$, then
$R_{t+1}$ is added at a uniformly chosen coordinate of $x$ and
$X_{t+1}^{(n)}=\sum_{i=1}^n x_i \oplus R_{t+1}\in \{0,1\}$ with
probability $1/2$ each, conditioned on $\F_t$. We now describe a
recursive relation for $W_t$, the Hamming weight of $X_t$:

\begin{equation}\label{eq:4.1}
  W_{t+1} = \sum_{j=2}^n \left( X_{t}^{(j)}\cdot 1_{(U_{t+1}\neq j)} +
    \left( X_{t}^{(j)}\oplus R_{t+1}\right) \cdot 1_{(U_{t+1} = j)}
  \right)+ 1_{\left( X_{t+1}^{(n)}=1\right)} \,.
\end{equation}

The first terms in \eqref{eq:4.1} follows from the fact that for
$1\leq j\leq n-1$,
$$
X_{t+1}^{(j)}=\begin{cases}
  X_t^{(j+1)} \hspace{1cm}\;\;\;\,\text{ if $U_{t+1}\neq j$} \\
  X_t^{(j+1)}\oplus R_{t+1} \;\;\text{ if $U_{t+1}=j$}
\end{cases}
$$

Taking conditional expectation in \eqref{eq:4.1}, we get

\begin{align}
  \E[W_{t+1} \mid \F_t] & = \sum_{j=2}^n\left( X_{t}^{(j)} \left( \frac{n-1}{n}\right) 
                          + \frac{1}{2}\left[1 - X_{t}^{(j)}+X_t^{(j)}\right] \frac{1}{n} \right) + \frac{1}{2} \nonumber \\
                        &= \left(1 - \frac{1}{n}\right)\sum_{j=2}^n X_t^{(j)}
                          + \left(\frac{n-1}{2n}\right) + \frac{1}{2} \nonumber \\
                        &= \left(1 - \frac{1}{n} \right)\sum_{j=1}^n X_t^{(j)}
                          - \left( 1 - \frac{1}{n}\right) X_t^{(1)}+ \left(\frac{2n-1}{2n}\right) \nonumber \\
                        &= \left( 1 - \frac{1}{n}\right) W_t - \left( 1-\frac{1}{n}\right) X_t^{(1)}+\frac{2n-1}{2n} \label{eq:4.2}
\end{align}

Let $\mu_t:=\E(W_t)$. Taking total expectation in \eqref{eq:4.2}, we
get

\begin{equation}\label{eq:4.3}
  \mu_{t+1}=\left( 1 - \frac{1}{n}\right)\mu_t-\left( 1 - \frac{1}{n}\right)\P\left( X_t^{(1)}=1 \right) + \frac{2n-1}{2n}
\end{equation}
We now estimate the probability in \eqref{eq:4.3}. Since $X_0 = 0$, for $t \leq n$,
\begin{equation}\label{eq:4.4}
  \P\left( X_t^{(1)} = 1 \right) = \left[ 1 - \left( 1-\frac{1}{n}\right) ^t \right]\frac{1}{2}  \,.
\end{equation}
To obtain \eqref{eq:4.4}, follow the bit at coordinate $1$ at time $t$
backwards in time: at time $t-1$ it was at coordinate $2$, at time
$t-2$ it was at coordinate $3$, etc.  At time $t$ it is a $0$ unless
it was updated at least once along this progression to the left, and
the last time that it was updated, it was updated to a $1$.
%
The probability it was never updated along its trajectory is
$\left(1-\frac{1}{n}\right)^t$, as we require that coordinates
$2,3,\cdots,t,t+1$ at times $t,t-1,\cdots,2,1$ respectively have not
been chosen for updates.  The probability is thus
$\left[ 1 - \left( 1-\frac{1}{n}\right) ^t \right]$ that at least one
of these coordinates is chosen; the factor of $1/2$ appears because we
need the last update to be to $1$.  Each update is independent of the
chosen coordinate and the previous updates.

We now look at a recursive relation for $\mu_t$.

\begin{equation}\label{eq:4.5}
  \mu_{t+1}=C_1\mu_t-\frac{C_1}{2}\left[ 1-C_1^t\right]+C_2 \,,
\end{equation}
where \eqref{eq:4.5} is obtained by plugging \eqref{eq:4.4} in
\eqref{eq:4.3} and the constants are
$C_1:=\left( 1-\frac{1}{n}\right), C_2:=\left(\frac{2n-1}{2n}\right)$.
Note that $\mu_0 = 0$.  The following Lemma obtains a solution of
\eqref{eq:4.5}.

\begin{lem}\label{lemma:4.1}
  $\mu_t=\left(\frac{t-n}{2}\right)\cdot C_1^{t}+\frac{n}{2}$ is the
  solution of the recursive relation \eqref{eq:4.5}.
\end{lem}
\begin{proof}
  Clearly $\mu_0 = 0$, and the reader can check that $\mu_t$ obeys the
  recursion.
%


\end{proof}

Note that we can write
$W_t=g_t(R_1,R_2,\ldots R_t,U_1,U_2,\ldots U_t)$ for some function
$g_t$, and that the variables $R_1, \ldots, R_t, U_1, \ldots, U_t$ are
independent.  The following Lemma is used in proving that $W_t$ has
variance of order $n$.

\begin{lem}
  \label{lemma:4.2}
$$\max_{\substack{r_1,\ldots,r_t\\u_1,\ldots, u_t}} \Big|g_t(r_1,\ldots,r_i,\ldots,r_t,u_1,\ldots, u_t)-g_t(r_1,\ldots,r_i\oplus 1,\ldots,r_t,u_1,\ldots u_t)\Big|\leq 2 \,.$$

for $1\leq i\leq t$ and $1\leq t\leq n$.
\end{lem}

\begin{proof}
  Any sequence of coordinates $\{u_s\}_{s=1}^t$ in
  $\{1,\ldots,n\}$ and bits
  $\{r_s\}_{s=1}^t$ in $\{0,1\}$ determine inductively a sequence $\{x_s\}_{s=0}^t$ in
  $\{0,1\}^n$ by updating at time $s$ the configuration $x_s$
  by adding the bit $r_s$ at coordinate $u_s$ followed by an application of
  the transformation $f$.   We call a sequence of
  pairs $\{(u_s,r_s)\}_{s=1}^t$ a \emph{driving sequence}, and the
  resulting walk in the hypercube $\{x_s\}_{s=0}^t$ the
  \emph{configuration sequence}.
  We write
  \[
    g_t = g_t(r_1,\ldots,r_t, u_1,\ldots, u_t), \qquad
    g_t' = g_t(r_1',\ldots,r_t', u'_1,\ldots,u_t')
  \]
  
  Consider a specific driving sequence of locations and update bits,
  $\{(r_s,u_s)\}_{s=1}^t$, and a second such
  \emph{driving sequence} $\{(r_s',u_s')\}_{s=1}^t$ which satisify
\begin{itemize}
\item $u_s'=u_s$ for $1 \leq s \leq t$,
\item $r_s' = r_s$ for $1 \leq s \leq t$ and $s \neq s_0$,
\item $r_{s_0}' = r_{s_0} \oplus 1$.
\end{itemize}
Thus, the two driving sequences agree everywhere except for at time $s_0$, where the update bits differ.

We want to show that $|g_t - g_t'| \leq 2$, for any $t \leq n$.

Let $\{x_s\}_{1 \leq s \leq t}$ and
$\{y_s\}_{1 \leq s \leq t}$ be the two configuration sequences
in $\{0,1\}^n$ obtained, respectively,
from the two driving sequences.  We will show inductively that the Hamming distance 
  \[
    d_s := d_s(x_s,y_s) := \sum_{j=1}^n |x_s^{(j)} -
    y_s^{(j)}|
  \]
  satisfies $d_s \leq 2$ for $s \leq t$, and hence the
  maximum weight difference $|g_t - g_t'|$ is bounded by $2$.
  
  Clearly $x_s = y_s$ for $s < s_0$, since the two
  driving sequences agree prior to time $s_0$, whence
  $d_s = 0$ for $s < s_0$.

  We now consider $d_{s_0}$.  Let
  $\ell = u_{s_0} = u'_{s_0}$ be the coordinate
  updated in both $x_{s_0}$ and $y_{s_0}$, and
  as before let
  \begin{align*}
      x_{s_0-1}' & =
      (x_{s_0-1}^{(1)}, \ldots, x_{s_0-1}^{(\ell)}\oplus r_{s_0}, \ldots, x_{s_0-1}^{(n)}) \,,\\
      y_{s_0-1}' & =
      (y_{s_0-1}^{(1)}, \ldots, y_{s_0-1}^{(\ell)}\oplus r'_{s_0}, \ldots, y_{s_0-1}^{(n)}) \,.
  \end{align*}
  Since $r_{s_0} \neq r'_{s_0}$ but
  $x_{s_0-1} = y_{s_0-1}$, the configurations
  $x_{s_0-1}'$ and $y_{s_0-1}'$ have different
  parities.   Recalling that 
  $x_{s_0} = f(x_{s_0-1}')$ and $y_{s_0} = f(y_{s_0-1}')$, we consequently have that
  $x_{s_0}^{(n)} \neq y_{s_0}^{(n)}$.
  Since $x_{s_0}$ and $y_{s_0}$ agree at all
  other coordinates except at $\ell-1$,
  we have
  \[
  d_{s_0} \leq I\{ \ell \neq 1\} + 1 \leq 2 \,.
  \]

  Next suppose that $d_s = 1$ for some time $s \geq s_0$, so that for some
  $\ell \in \{1,\ldots,n\}$, we have 
  $x_s^{(j)} = y_s^{(j)}$ for $j \neq \ell$ and
  $x_s^{(\ell)} \neq y_s^{(\ell)}$.  Since $r_{s+1} = r'_{s+1}$
  and $u_{s+1} = u'_{s+1}$,  after adding the same update bit
  at the same coordinate in the configurations $x_s$ and $y_s$,
  but before applying $f$,
  the resulting configurations will still have a single disagreement
  at $\ell$.  Thus, after applying $f$ to obtain the configurations at
  time $s+1$, we have
  $x_{s+1}^{(n)} \neq y_{s+1}^{(n)}$, but
  \[
    d_{s+1} = \sum_{j=1}^{n-1}|x_{s+1}^{(j)} - y_{s+1}^{(j)}|
    + |x_{s+1}^{(n)} - y_{s+1}^{(n)}| 
    = \sum_{j=2}^n |x_{s}^{(j)} - y_{s}^{(j)}| + 1
    \leq d_s + 1 \leq 2 \,.
  \]
  (If $\ell = 1$, then $d_{s+1} = 1$.)   Thus, $d_{s+1} \leq 2$.

  Finally consider the case that $d_s = 2$
  for $s \geq s_0$. Again, $u_{s+1} = u'_{s+1}$ and $r_{s+1} =
  r'_{s+1}$; After updating $x_s$ and $y_s$ with the same
  bit at the same coordinate, but before
  applying $f$, the two configurations still differ at
  exactly these two coordinates.
  Thus, $x_{s+1}^{(n)} = y_{s+1}^{(n)}$, and
  \[
    d_{s+1} = \sum_{j=1}^{n-1}|x_{s+1}^{(j)} - y_{s+1}^{(j)}|
    + 0 = \sum_{j=2}^{n-1} |x_s^{(j)} - y_s^{(j)}|
    \leq d_s \leq 2 \,.
  \]
  (Again, the sum is $1$ if one of the two disagreements at
  time $s$ is at coordinate $1$.)

  We now have that $d_s \leq 2$ for all $s \leq t$:  For $s \leq s_0$,
  we have $d_s = 0$, and $d_{s_0} = 1$.   For $s \geq s_0$, if
  $d_{s} \leq 2$, then $d_{s+1} \leq 2$.
  It then follows in particular that $d_t \leq 2$ and that
  $|g_t - g_t'| \leq 2$.
\end{proof}
\begin{lem}
  \label{lemma:4.3}
$$\max_{\substack{r_1,\ldots,r_t\\
    u_1,\ldots,u_t, u_i'}}
\Big|g_t(r_1,\ldots,r_t,u_1,\ldots,u_i,\ldots,
u_t)-g_t(r_1,\ldots,r_t,u_1,\ldots,u_i',\ldots, u_t)\Big|\leq 2 \,.$$
\end{lem}

\begin{proof}
  Again, if two trajectories differ only in the coordinate selected at
  time $i$, then the weight at time $t$ can by differ by at most
  $2$. Fix time $1\leq t\leq n$ and consider the dynamics of number of
  coordinates at which the two trajectories differ at time $k<t$.

  The two trajectories agree with each other until time $i$ because
  the same random bits and locations are used to define these
  trajectories.  At time $i$, we add the same random bit $r_i$ to
  update both trajectories, but use coordinate $u_i$ for the first
  trajectory and coordinate $u_i'$ in the second trajectory. If
  $r_i=0$, then clearly the two trajectories continue to agree at time
  $k\geq i$.

  Now suppose that $r_i=1$. Let $b_1,b_2$ be the bits at coordinates
  $u_i,u_i'$ in the \emph{first} trajectory at time $i-1$ and
  $b_3,b_4$ be the bits at coordinates $u_i,u_i'$ in the \emph{second}
  trajectory at time $i-1$.  Note that since the trajectories are
  identical for times less than $i$, $b_1 = b_3$ and $b_2=b_4$.  For
  all values of $(b_1,b_2,b_3,b_4)$ satisfying $b_1 = b_3$ and
  $b_2 = b_4$, there are two disagreements between the trajectories at
  coordinates $u_i-1,u_i'-1$ at time $i$.  (If $u_i-1 < 0$ or
  $u_i' - 1 < 0$, then there is a single disagreement).  The appended
  bit agrees, since
  \[
    (b_1\oplus 1) \oplus b_2 = b_1 \oplus (b_2 \oplus 1) = b_3 \oplus
    (b_4 \oplus 1) \,.
  \]
  This takes care of what happens at time $i$ when the single
  disagreement between update coordinates occurs; at time $i$ the
  Hamming distance is bounded by $2$.

  Now we consider an induction on the Hamming distance, showing that
  at all times the Hamming distance is bounded by two.

  {\em Case A}.  Suppose that the trajectories differ at two
  coordinates, say $\ell_1,\ell_2$ at time $k>i$.  Since the two
  trajectories only differ in the updated coordinate at time $i$ with
  $i < k$, the chosen update coordinate and the chosen update bit are
  the same for both trajectories at time $k$.  Let $b_1,b_2$ be the
  bits at coordinates $\ell_1,\ell_2$ in the \emph{first} trajectory
  at time $k$ and $b_3,b_4$ be the bits at coordinates $\ell_1,\ell_2$
  in the \emph{second} trajectory at time $k$. Necessarily
  $b_1 \neq b_3$ and $b_2 \neq b_4$.  Consider the following subcases:

  {\em Subcase 1.} $(b_1,b_2,b_3,b_4)=(0,1,1,0)$ or
  $(b_1,b_2,b_3,b_4)=(1,0,0,1)$. If $u_k\notin\{\ell_1,\ell_2\}$, then
  the trajectories continue to have the same two disagreements at
  these coordinates shifted by one at time $k+1$, since the updated
  coordinate and the update bit is the same for both trajectories.
  Also, the new bit which is appended agrees, since
  $b_1 \oplus b_2 = 1 = b_3 \oplus b_4$, and all other bits in the
  mod-$2$ sum agree. So the Hamming distance remains bounded by two,
  allowing the possibility the Hamming distance decreases if
  $\ell_1 \wedge \ell_2 = 1$.

  Supposing that $u_k \in \{\ell_1,\ell_2\}$, without loss of
  generality, assume that $u_k = \ell_1$.  These disagreements
  propagate to time $k+1$, allowing for the possibility for one to be
  eliminated if it occurs at coordinate $1$. For $r_k = 0$, the
  appended bit will agree since
  \[
    (b_1\oplus 0) \oplus b_2 = 1 = (b_3 \oplus 0) \oplus b_4
  \]
  and all other bits in the mod-$2$ sum agree.  For $r_k = 1$, the
  appended bit will still agree since
  \[
    (b_1 \oplus 1) \oplus b_2 = 1 = (b_3 \oplus 1) \oplus b_4 \,.
  \]
  This means at time $k+1$, the Hamming distance is bounded by $2$.


  {\em Subcase 2.} $(b_1,b_2,b_3,b_4)=(1,1,0,0)$ or
  $(b_1,b_2,b_3,b_4)=(0,0,1,1)$. If $u_k\notin\{\ell_1,\ell_2\}$, then
  the trajectories continue to have the same two disagreements (unless
  one of the disagreements is at coordinate $1$), the appended bit
  agrees (since $1\oplus 1 = 0 \oplus 0$), and the Hamming distance
  remains bounded by $2$.

  Suppose that $u_k = \ell_1$.  If $r_k = 0$, then the two
  disagreements persist (or one is eliminated because it occurred at
  coordinate $1$) and the appended bit agrees.  If $r_k=1$, then the
  two disagreements persist, and again the appended bit agrees,
  because now $0 \oplus 1 = 1 \oplus 0$,

  Therefore the Hamming distance remains bounded by $2$ at time $k+1$.


  {\em Case B}.  Suppose that the trajectories differ at one
  coordinate, say $\ell$, at time $k>i$. Consider the following
  subcases:

  {\em Subcase 1.} $u_k\neq \ell$. The disagreement persists unless
  $u_k = 1$, and the appended bit now disagrees.  Thus the Hamming
  distance is bounded by $2$ at time $k+1$.

  {\em Subcase 2.} $u_k=\ell$.  The disagreement persists at $u_k$
  (unless $u_k = 1$), and the appended bit now disagrees. Again, the
  Hamming distance is bounded by $2$ at time $k+1$.

  Thus, by induction, the Hamming distance between the two
  trajectories remains always bounded by $2$.  As a consequence, the
  difference in the Hamming {\em weight} remains never more than $2$.
\end{proof}

\begin{lem} \label{lem:BDI} If $W_0=0$, then $\Var(W_t) \leq 4t$.
\end{lem}
\begin{proof}
  We use the following consequence of the Efron-Stein inequality:
  Suppose that $g:{\mathcal X}^n \to \R$ has the property that for
  constants $c_1,\ldots,c_n > 0$,
  \[
    \sup_{x_1,\ldots,x_n, x_i'} |g(x_1, \ldots, x_n) -
    g(x_1,\ldots,x_{i-1},x_i',x_{i+1},\ldots, x_n)| \leq c_i \,.
  \]
  and if $X_1,\ldots,X_n$ are independent variables, and
  $Z=g(X_1,\ldots,X_n)$ is square-integrable, then
  $\Var(Z) \leq 4^{-1}\sum_i c_i^2$.  (See, for example,
  \textcite[Corollary 3.2]{CI}.)

  This inequality together with Lemmas \ref{lemma:4.2} and
  \ref{lemma:4.3} show that
  \begin{equation}
    \Var(W_t)\leq \frac{1}{2}\sum_{i=1}^{2t} 2^2 = 4t \leq 4n \quad\text{for $t\leq n$}
  \end{equation}
\end{proof}

\begin{proof}[Proof of Theorem \ref{thm:main}(ii)]
  Plugging $t=n-n^\alpha$ in Lemma \ref{lemma:4.1}, where
  $\frac{1}{2}<\alpha<1$, we get
  \begin{equation}
    \E W_t=\mu_t=\frac{n}{2}-\left( 1-\frac{1}{n}\right)^{n-n^{\alpha}}\frac{n^\alpha}{2}\leq \frac{n}{2}-\frac{1}{2e}n^\alpha \label{eq:4.18}
  \end{equation}

  For any real-valued function $h$ on $S$ and probability $\mu$ on
  $S$, write $E_{\mu}h:=\sum_{x\in S}h(x)\mu (x)$.  Similarly,
  $\Var_\mu(h)$ is the variance of $h$ with respect to $\mu$.  As
  stated earlier, $W(x)$ is the Hamming weight of $x\in S$. The
  distribution of the random variable $W$ under the stationary
  distribution $\pi$ (uniform on $\{0,1\}^n$), is binomial with
  parameters $n$ and $1/2$, whence
  \begin{equation} \label{eq:statmom} E_\pi(W)=\frac{n}{2}, \quad
    \Var_{\pi}(W)=\frac{n}{4} \,.
  \end{equation}

  Let $c>0$ be a constant and
  $A_c:=\left(\frac{n}{2}-c\sqrt{n},\infty\right)$.  Chebyshev's
  inequality yields that
  \[
    \pi\{ W \in A_c\} \geq 1 - \frac{1}{4 c^2} \,.
  \]
  Thus we can pick $c$ so that this is at least $1 - \eta$ for any $\eta>0$.


  Fix $\frac{1}{2}< \alpha< 1$.  For $t_n=n-n^\alpha$, by
  \eqref{eq:4.18},
  \begin{align*}
    \P_0\left(W_{t_n}\in A_c \right) & =
                                       \P_0\left( W_{t_n} > \frac{n}{2} - c\sqrt{n}\right) \\
                                     & \leq  \P_0\left(W_{t_n}-\E W_{t_n}\geq \frac{n^\alpha}{2e}-c\sqrt{n}\right) \,.
  \end{align*}
  Since
  \[
    \frac{n^\alpha}{2e} - c n^{1/2} = n^{1/2}\underbrace{\left(
        \frac{n^{\alpha-1/2}}{2e} - c \right)}_{\delta_n(c)}\,,
  \]
  we have again by Chebyshev's inequality,
  \[
    \P_0(W_{t_n} \in A_c) \leq \frac{ \Var(W_{t_n})}{n\delta_n(c)^2}
    \leq \frac{4 t_n}{n\delta_n(c)^2} \leq \frac{4}{\delta_n(c)^2} \,.
  \]
  The last inequality follows from Lemma \ref{lem:BDI}, since
  $t_n \leq n$.
%

  Finally,
  \[
    \|Q_1^{t_n}(\0,\cdot)-\pi\|_{\text{TV}}\geq \big|\pi(W\in A_c)-
    \P_0\left(W_{t_n}\in A_c \right)\big| \geq 1 - \frac{1}{4c^2} -
    \frac{4}{\delta_n(c)^2} \,.
  \]
  One can take, for example, $c_n = \log n$ so that
  $\delta_n(c_n) \to \infty$, in which case the bound above is
  $1 - o(1)$.

\end{proof}

\section{A related chain}
\label{section:6}

We now consider a Markov chain $Q_2$ on $\{0,1\}^{2m}$ related to the
the chain $Q_1$.  One step of $Q_2$ chain again consists of combining a
stochastic move with $f$.  Instead of updating a random coordinate,
now the coordinate is always the ``middle'' coordinate.
Thus, when at $x$, first the random move
\[
  x \mapsto x' = (x_1,x_2.\cdots,x_{m}\oplus R,x_{m+1},\cdots,x_{2m})
\]
where $R$ is a an independent random bit.  Afterwards, again the
transformation $f$ is applied to yield the new state $f(x')$.

\begin{theorem}\label{t2}
  For all $x$, $\|Q_2^{(n)}(x,\cdot)-\pi\|_{{\rm TV}}= 0$ for all $n=2m$ where $m\geq 1$.
\end{theorem}

\begin{rmk}
  Note that if the transformation is the circular shift instead of
  $f$, then this would be equivalent to systematic scan, which
  trivially yields an exact uniform sample in exactly $n=2m$ steps.
  Thus this chain can be viewed as a small perturbation of systematic
  scan which is Markovian and still yields an exact sample in $n=2m$
  steps.
\end{rmk}

\begin{proof}
We denote by $(R_1, R_2, \ldots)$ the sequence of bits used to update
the chain.

To demonstrate how the walk evolves with time by means of an example,
Table \ref{table:1} shows the coordinates of $Y_t$ at different $t$
for $2m=6$, when starting from $\0$.

\begin{table}[h!]
  \centering
  \begin{tabular}{||c c c c c c c||}
    \hline
    Coordinate & 1 & 2 & 3 & 4 & 5 & 6 \\ [1ex] 
    \hline\hline
    t=0 & 0 & 0 & 0 & 0 & 0 & 0 \\ 
    t=1 & 0 & $R_1$ & 0 & 0 & 0 & $R_1$ \\
    t=2 & $R_1$ & $R_2$ & 0 & 0 & $R_1$ & $R_2$ \\
    t=3 & $R_2$ & $R_3$ & 0 & $R_1$ & $R_2$ & $R_3$ \\
    t=4 & $R_3$ & $R_4$ & $R_1$ & $R_2$ & $R_3$ & $R_1\oplus R_4$ \\ 
    t=5 & $R_4$ & $R_1\oplus R_5$ & $R_2$ & $R_3$ & $R_1\oplus R_4$ & $R_2\oplus R_5$\\
    t=6 & $R_1\oplus R_5$ & $R_2\oplus R_6$ & $R_3$ & $R_1\oplus R_4$ & $R_2\oplus R_5$ & $R_3\oplus R_6$ \\ [1ex]
    \hline
  \end{tabular}
  \caption{Evolution of coordinates with time for $2m=6$}
  \label{table:1}
\end{table}


Let $n = 2m$ be an even integer, and let $Z_1,Z_2,\cdots ,Z_{m}$ be
the random variables occupying the $n$ coordinates at time
$t=n$. 
The following relationships hold for any starting state
$x=(x_1,x_2,\ldots,x_{n})$:

\begin{align*}
  Z_1 &= R_1\oplus R_{m+2}\oplus \bigoplus_{i=1}^{2m} x_i  \\
  Z_2 &= R_2\oplus R_{m+3}\oplus x_1 \\
      &\vdots \\
  Z_{m-1} &= R_{m-1}\oplus R_{2m}\oplus x_{m-2}  \\
  Z_m &= R_n\oplus x_{m-1} \\
  Z_{m+1} &= R_1\oplus R_{m+1}\oplus x_m \\
  Z_{m+2} &= R_2\oplus R_{m+2}\oplus x_{m+1} \\
      &\vdots \\
  Z_{2n} &= R_n\oplus R_{2n}\oplus x_{2n-1} 
\end{align*}

This is because at $t=1$, the random variable at coordinate $n$ is
$R_1+\bigoplus_{i=1}^{n} x_i$. At time $t=m+1$, this random variable
moves to coordinate $m$ because of successive shift register
operations. Because the coordinate updated at any time along this
chain is $m$, we have that at time $t=m+2$, the random variable at
coordinate $m-1$ is $R_1+R_{m+2}\oplus \bigoplus_{i=1}^{n}
x_i$. Again, because of successive shift register operations, the
random variable $R_1+R_{n+2}\oplus \bigoplus_{i=1}^{n} x_i$ moves to
coordinate $1$ at time $n$. The random variables at other coordinates
can similarly be worked out. Thus the above system of equations can be
written in matrix form as $Z=BR+\Vec{x}$ where:

\[
  Z = (Z_1,\ldots,Z_n)^T, \quad R = (R_1,\ldots,Z_n)^T \,,
\]
and

\begin{align*}
  B_{n\times n} & =\begin{bmatrix} I_{m\times m} & C_{m\times m} \\
    I_{m\times m} & I_{m\times m} \end{bmatrix}, &
                                                   C_{m\times m} & = \begin{bmatrix} 0_{(m-1)\times 1} & I_{(m-1)\times (m-1)} \\ 0_{1\times 1} & 0_{1\times (m-1)} \end{bmatrix} \,.
\end{align*}

and
\[
  \Vec{x}=\begin{bmatrix}\oplus_{i=1}^{n} x_i\\x_1\\\vdots
    \\x_{n-1} \end{bmatrix}
\]

Note that
\[
  \det (B)=\det (I) \times \det (I-II^{-1}C) = \det (I-C)=1\neq 0 \,.
\]
The last equality follows since $\det (I-C)=1$ because $I-C$ is an
upper triangular matrix with ones along the main diagonal. Hence $B$
is an invertible matrix and if $z\in \{0,1\}^{n}$, then
$$
\P(Z=z)=\P\big(R=B^{-1}(z-\Vec{x})\big)=\frac{1}{2^{n}} \,,
$$
where the last equality follows from the fact that $R$ is uniform over
$S=\{0,1\}^{n}$. Thus the state along $Q_2$ chain at $t=2m=n$ is
uniform over $S$ and
\[
  \|Q_2^{(n)}(\0,\cdot)-\pi\|_{\text{TV}}=0, \quad \text{$n$ is even}.
\]

\end{proof}

\section{Conclusion and Open Questions}

Here we have shown that the ``shift-register'' transformation speeds
up the mixing of the walk on the hypercube, for which the stationary
distribution is uniform.  The shift-register transformation is a good
candidate for a deterministic mixing function, as shift-registers were
used for early pseudo-random number generation.


One of our original motivations for analyzing this chain was in part that the
uniform distribution corresponds to the infinite temperature Ising
measure.  Indeed, of great interest are chains on product spaces
having non-uniform stationary distributions, such as Gibbs measures. Finding deterministic transformations which speed up mixing for non-uniform distributions remains a challenging and intriguing open problem.
%

\section*{Acknowledgements}

The authors thank the reviewer for helpful comments and pointing out relevant references.

\printbibliography 
\end{document}